\documentclass[a4paper]{amsart} 
\usepackage{amssymb}  
\usepackage{amsmath}  
\usepackage{amscd}
\usepackage{graphicx} 
\title{Cobordism of exact links}
\author{Vincent Blanl\oe il}
\address{Universit\'e de Strasbourg,
IRMA,
7, rue Ren\'e Descartes,
67084 Strasbourg cedex, 
France}
\email{v.blanloeil@math.unistra.fr}
\author{Osamu Saeki}
\address{Institute of Mathematics for Industry, Kyushu University,
Motooka 744, Nishi-ku, Fukuoka 819-0395,
Japan}
\email{saeki@imi.kyushu-u.ac.jp}
\thanks{The second author has been partially supported by the 
University of Strasbourg, France,
by FY 2008 
Researcher Exchange Program between JSPS and CNRS, and
by Grant-in-Aid for Scientific Research
(No.~23244008), JSPS}
\newtheorem{thm}{Theorem}[section]

\newtheorem{lem}[thm]{Lemma}
\newtheorem{prop}[thm]{Proposition}
\theoremstyle{definition}
\newtheorem{dfn}[thm]{Definition}

\theoremstyle{remark}
\newtheorem{exam}[thm]{Example}
\newtheorem{rmk}[thm]{Remark}
\newtheorem{ask}[thm]{Question}
\makeatletter
 
 \@addtoreset{equation}{section}
\makeatother

\newcommand{\Z}{\mathbf{Z}}

\newcommand{\Int}{\mathop{\mathrm{Int}}\nolimits}
\newcommand{\Ker}{\mathop{\mathrm{Ker}}\nolimits}
\newcommand{\Coker}{\mathop{\mathrm{Coker}}\nolimits}
\newcommand{\rk}{\mathop{\mathrm{rank}}\nolimits}
\newcommand{\Hom}{\mathop{\mathrm{Hom}}\nolimits}

\newcommand{\Tors}{\mathop{\mathrm{Tors}}\nolimits}
\newcommand{\Image}{\mathop{\mathrm{Im}}\nolimits}

\renewcommand{\tilde}{\widetilde}
\renewcommand{\setminus}{\smallsetminus}
\def\spmapright#1{\smash{%
 \mathop{\hbox to 1cm{\rightarrowfill}}
  \limits^{#1}}}
\def\spmapleft#1{\smash{%
 \mathop{\hbox to 1cm{\leftarrowfill}}
  \limits^{#1}}}

\begin{document}
\begin{abstract}
In this paper, we %correct some of the statements and proofs 
%define 
introduce the notion of 
%\emph{exact links}, 
exact links,
%which is the most general topological
%definition to study cobordism and 
which constitute a broad class of high dimensional links, 
in order to clarify the results 
obtained in the article ^^ ^^ V.~Blanl\oe il,
Cobordisme des entrelacs,
Ann.\ Fac.\ Sci.\ Toulouse Math.\ (6) \textbf{7} (1998), 
185--205" concerning cobordisms of odd
dimensional non-spherical links. %by introducing the new notion of exact Seifert surfaces.
\end{abstract}

\maketitle 

\section{Introduction}\label{sect:intro}

In \cite{B2}, a lot of important results have
been obtained concerning cobordisms of odd
dimensional non-spherical links. Unfortunately,
%the paper contains some errors and 
some statements must be clarified, since there are even counter examples.
%The error occurred, since the argument there started with
In fact, when studying knot cobordisms it is 
%not possible 
very difficult
to begin with
an arbitrary Seifert surface of a given link. For example, if we
consider $S^{n-1} \times D^{n+1}$ embedded trivially
in $S^{2n+1}$ as a Seifert surface of its boundary
link $K \cong S^{n-1} \times S^n$, then the associated
Seifert form is represented by the $0 \times 0$ matrix,
since $H_n(S^{n-1} \times D^{n+1}; \Z)$ vanishes.
If there existed a simple link $K'$, i.e.\ a link $K'$
admitting an $(n-1)$-connected Seifert surface, which is
cobordant to $K$ and with algebraically cobordant
Seifert form (for details, see \S\ref{preliminaries}),
then $K'$ would necessarily
be diffeomorphic to the sphere $S^{2n-1}$.
This gives a simple counter example to
\cite[Propositions~2.1 and 2.2]{B2}.

In this paper, we will instead use a 
reasonably broad class of Seifert surfaces, 
called 
%\emph{exact Seifert surfaces},
exact Seifert surfaces,
and obtain similar results.

Throughout the paper, we work
in the smooth category.
All homology and
cohomology groups are with integer coefficients.
The symbol ^^ ^^ $\cong$" denotes a diffeomorphism
between manifolds
or an appropriate isomorphism between algebraic objects.

The second author would like to express his thanks to
the people at IRMA, %Louis Pasteur 
University of Strasbourg,
France, for their hospitality during the preparation of
the manuscript.

\section{Preliminaries}\label{preliminaries}

Let us first recall some definitions.

\begin{dfn}\label{dfn:knot}
Let $K$ be a closed $(2n-1)$-dimensional manifold embedded
in the $(2n+1)$-dimensional sphere $S^{2n+1}$.
We suppose that $K$ is 
$(n-2)$-connected.
(We adopt the convention that a space
is $(-1)$-connected if it is not empty.)
We further assume that it is oriented.
Then we call $K$ or its (oriented) isotopy class
a \emph{$(2n-1)$-link}, or
simply a \emph{link}.

A link $K$ is a \emph{knot} if $K$ is 
a homotopy $(2n-1)$-sphere.
\end{dfn}

\begin{dfn}\label{dfn:cob}
Two $(2n-1)$-links $K_0$ and $K_1$ in $S^{2n+1}$ are 
said to be \emph{cobordant} if there exists a 
properly embedded oriented $2n$-dimensional
manifold $X$ of $S^{2n+1} \times [0,1]$ such that
\begin{enumerate}
\item $X$ is diffeomorphic to $K_0 \times [0,1]$, and
\item $\partial X = (-K_0 \times \{0\}) \cup (K_1 \times \{1\})$,
\end{enumerate}
where $-K_0$ is obtained from $K_0$ by reversing
the orientation.
\end{dfn} 

It is known that for every $(2n-1)$-link $K$, 
there exists a compact oriented $2n$-dimensional
submanifold $F$ of $S^{2n+1}$ having $K$ as 
boundary. 

\begin{dfn}
Let $K$ be a $(2n-1)$-link.
A compact oriented $2n$-dimensional submanifold $F$ of 
$S^{2n+1}$
having $K$ as its oriented boundary is called a 
\emph{Seifert surface} associated with $K$.
\end{dfn}

\begin{dfn}
We say that a $(2n-1)$-link is \emph{simple} if 
it admits an $(n-1)$-connected Seifert surface.
\end{dfn}

\begin{dfn}\label{dfn:fibered}
We say that a $(2n-1)$-link $K$ is \emph{fibered}
if there exists a smooth
fibration 
$$\phi : S^{2n+1} \setminus K \to S^1$$
and a trivialization 
$$\tau : N(K) \to K \times D^2$$
of a closed tubular neighborhood $N(K)$ of $K$
in $S^{2n+1}$ such that $\phi|_{N(K) \setminus K}$
coincides with $\pi \circ \tau|_{N(K) \setminus K}$,
where $\pi : K \times (D^2 \setminus \{0\}) \to
S^1$ is the composition of the projection to the
second factor and the obvious projection
$D^2 \setminus \{0\} \to S^1$.

Note that then the closure of each fiber of $\phi$
in $S^{2n+1}$
is a compact $2n$-dimensional oriented manifold
whose boundary coincides with $K$.
We shall often call the 
closure of each fiber simply a \emph{fiber}.

Furthermore, we say that a fibered $(2n-1)$-link $K$ is
\emph{simple} if each fiber of $\phi$ is $(n-1)$-connected.
\end{dfn}

\begin{dfn}
Suppose that $F$ is a compact oriented $2n$-dimensional 
submanifold of
$S^{2n+1}$, and let $G$ be the quotient of $H_n(F)$ by its
$\Z$-torsion. 
The \emph{Seifert form} associated with $F$ is the bilinear form
$$A : G \times G \to \Z$$ 
defined as follows. For $(x,y)
\in G \times G$, we define $A(x,y)$ to be the linking number 
in $S^{2n+1}$ of $\xi_+$ and $\eta$, where $\xi$ and $\eta$ are 
$n$-cycles
in $F$ representing $x$ and $y$ respectively, and
$\xi_+$ is the $n$-cycle $\xi$ pushed off $F$ into 
the positive normal direction
to $F$ in $S^{2n+1}$.

By definition a \emph{Seifert form} associated with a
$(2n-1)$-link $K$ is the Seifert form
associated with $F$, where $F$ is a Seifert 
surface associated with $K$.
A matrix representative of a Seifert form
with respect to a basis of $G$ is called
a \emph{Seifert matrix}.
\end{dfn}

\begin{dfn}\label{dfn1}
Let $\mathcal{A}$ be the set of all bilinear forms defined on 
free $\Z$-modules $G$ of finite rank. Set $\varepsilon = (-1)^n$.
For $A \in \mathcal{A}$, let us denote by $A^T$ the transpose of $A$,
by $S$ the $\varepsilon$-symmetric form $A+\varepsilon A^T$ 
associated with $A$, by $S^\ast : G \rightarrow G^\ast$ the 
adjoint of $S$ with $G^\ast$ being the dual 
$\Hom_\Z(G, \Z)$ of $G$, and by 
$\overline{S} : \overline{G} \times \overline{G}
\rightarrow \Z$ the 
$\varepsilon$-symmetric nondegenerate form induced by $S$ 
on $\overline{G} = G/\Ker{S^\ast}$. 
A submodule $M$ of $G$ is said to be \emph{pure} 
if $G/M$ is torsion free, or equivalently if $M$ is a direct
summand of $G$. For a submodule
$M$ of $G$, let us denote by $M^{\wedge}$ the smallest pure 
submodule of $G$ that contains $M$. 
We denote by $\overline{M}$ the image of $M$ in $\overline{G}$
by the natural projection map.
\end{dfn}

\begin{dfn}\label{dfn2}
Let $A : G \times G \rightarrow \Z$ 
be a bilinear form in $\mathcal{A}$. The 
form $A$ is \emph{Witt associated to $0$} if the rank $m$ of 
$G$ is even and
there exists a pure submodule $M$ of rank $m/2$ in $G$ such that $A$ 
vanishes on $M \times M$. Such a submodule $M$ is called a \emph{metabolizer} 
for $A$.
\end{dfn}

\begin{dfn}\label{dfn3}
Let $A_i : G_i \times G_i \rightarrow \Z$, 
$i=0, 1$, be two bilinear forms in $\mathcal{A}$. 
Set $G = G_0 \oplus G_1$, $A = (-A_0) \oplus A_1$,
$S = A + \varepsilon A^T$,
and $S_i = A_i + \varepsilon A_i^T$, $i = 0, 1$. 
The form $A_0$ is said to be \emph{algebraically cobordant} to $A_1$
if there exist a metabolizer $M$ for $A$ such that 
$\overline{M}$ is pure in $\overline{G}$, an 
isomorphism $\varphi : \Ker{S_0^\ast} \rightarrow
\Ker{S_1^\ast}$, and an isomorphism $\theta : 
\Tors(\Coker{S_0^\ast}) \rightarrow
\Tors(\Coker{S_1^\ast})$ which satisfy the following 
two conditions: 
\begin{equation}
%\begin{itemize} 
%\item[(c1)] $
M \cap \Ker{S^\ast} = \bigl\{(x, \varphi(x))\,|\, 
x \in \Ker{S_0^\ast} \bigr\} \subset 
\Ker{S_0^\ast} \oplus \Ker{S_1^\ast}
= \Ker{S^\ast}, 
\tag{$c1$}
\end{equation}
\begin{align}
d(S^\ast(M)^{\wedge}) & = \bigl\{ (y, \theta(y))\,|\, y \in
\Tors(\Coker{S_0^\ast})\bigr\} \tag{$c2$} \\
& \subset \Tors(\Coker{S_0^\ast})
\oplus \Tors(\Coker{S_1^\ast}) = 
\Tors(\Coker{S^\ast}), 
\nonumber
\end{align}
where $d$ is the quotient map $G^\ast \rightarrow 
\Coker{S^\ast}$
and ^^ ^^ $\Tors$" means the torsion subgroup. 
In the above situation, we also say that
$A_0$ and $A_1$ are \emph{algebraically cobordant with respect to
$\varphi$ and $\theta$}. 
%Note that the algebraic cobordism is
%an equivalence relation as has been shown in \cite[Theorem~1]{BM}.
\end{dfn}

\begin{rmk}
As pointed out in \cite{BS2}, the relation of
algebraic cobordism may not be an equivalence
relation on the set of all
integral bilinear forms of finite rank
(see also \cite{V2}).
However, it is an equivalence relation on the
set of all \emph{unimodular} bilinear forms of finite
rank (see \cite{BM}).
\end{rmk}

\section{Exact links}

\begin{dfn}
Suppose $n \geq 2$.
A Seifert surface $F$ of a $(2n-1)$-link $K$
is said to be \emph{exact}
if the sequence
$$
0 \to H_n(K) \to H_n(F)/\Tors{H_n(F)} \to 
H_n(F, K)/\Tors{H_n(F, K)} \to 
H_{n-1}(K) \to 0,
$$
derived from the homology exact sequence
for the pair $(F, K)$,
is well defined and exact. 
%(or $$0 \to H_n(K) \to H_n(F) \to H_n(F, K)
%\to H_{n-1}(K) \to 0$$
%is exact modulo finite groups).
Note that the homomorphism 
$$H_n(F, K)/\Tors{H_n(F, K)} \to H_{n-1}(K)$$
may not be well defined in general.
Here, we impose the condition that this map should be well defined.
A $(2n-1)$-link
is said to be \emph{exact} if
it admits an exact Seifert surface.
\end{dfn}

\begin{exam}
Consider $K = S^{n-1} \times S^n$ embedded trivially in
$S^{2n} \subset S^{2n+1}$, $n \geq 2$. Then $K$ is a $(2n-1)$-link
and it bounds two Seifert surfaces $F_0 = D^n \times S^n$
and $F_1 = S^{n-1} \times D^{n+1}$, both of which are
embedded in $S^{2n}$. Then $F_0$ is exact, while
$F_1$ is not, since $H_n(S^{n-1} \times S^n)
\to H_n(S^{n-1} \times D^{n+1})$ is not a monomorphism.
\end{exam}

\begin{lem}\label{lem:exact}
For $n \geq 2$, we have the following.
\begin{itemize}
\item[$(1)$] A simple $(2n-1)$-link is always exact. In fact,
every $(n-1)$-connected Seifert surface is exact.
\item[$(2)$] A fibered $(2n-1)$-link is always exact.
In fact, every fiber is exact.
\item[$(3)$] A $(2n-1)$-knot is always exact.
In fact, every Seifert surface is exact.
\end{itemize}
\end{lem}

\begin{proof}
In the following, let $K$ be a $(2n-1)$-link and 
$F$ a relevant Seifert surface.

(1) Let us consider the exact sequence
$$H_{n+1}(F, K) \to H_n(K) \to H_n(F) \to
H_n(F, K) \to H_{n-1}(K) \to H_{n-1}(F).$$
Then we have the desired result, 
since $H_{n+1}(F, K) \cong H^{n-1}(F) = 0$, 
$H_{n-1}(F) = 0$, and
$H_n(F)$ and $H_n(F, K)$ are torsion free.

(2) If $F$ is a fiber of a fibered link, then it is easy to
see that $S^{2n+1} \setminus F$ is homotopy equivalent to $F$.
Hence, by Alexander duality, we have
$$\tilde{H}_i(F) \cong \tilde{H}^{2n-i}(F)$$
for all $i$, where $\tilde{H}_*$ and $\tilde{H}^*$
denote reduced homology and cohomology groups, respectively.
Consider the exact sequence
$$
\begin{array}{cccccc}
0 & \to & \tilde{H}_{n+1}(F) & \to & H_{n+1}(F, K) & \to \\
\tilde{H}_n(K) & \to & \tilde{H}_n(F) & \to & H_n(F, K) & \to \\
\tilde{H}_{n-1}(K) & \to & \tilde{H}_{n-1}(F) & \to & H_{n-1}(F, K) & \to 0.
\end{array}
$$
(Recall that $K$ is $(n-2)$-connected.)
Since 
$$\tilde{H}_{n-1}(F) \cong \tilde{H}^{n+1}(F) \cong H_{n-1}(F, K)$$
and $\tilde{H}_{n-1}(F) \to H_{n-1}(F, K)$ is an epimorphism,
it must be an isomorphism.
Hence $H_n(F, K) \to \tilde{H}_{n-1}(K)$ is an epimorphism.
Furthermore, since $\tilde{H}_{n+1}(F) \cong \tilde{H}^{n-1}(F)
\cong H_{n+1}(F, K)$, $\tilde{H}_{n+1}(F) \to H_{n+1}(F, K)$ is
a monomorphism, and $\tilde{H}_n(K)$ is torsion free,
the homomorphism 
$\tilde{H}_{n+1}(F) \to H_{n+1}(F, K)$ must be an isomorphism.
Thus $\tilde{H}_n(K) \to \tilde{H}_n(F)$ is a monomorphism.
Since $\tilde{H}_n(K)$ is torsion free, the map
$$\tilde{H}_n(K) \to \tilde{H}_n(F)/\Tors{\tilde{H}_n(F)}$$
is also a monomorphism.
Finally, since $\tilde{H}_n(F) \cong \tilde{H}^n(F)
\cong H_n(F, K)$, we have
$\Tors{\tilde{H}_n(F)} \cong \Tors{H_n(F, K)}$.
Then we see easily that
the sequence
$$
0 \to \tilde{H}_n(K) \to \tilde{H}_n(F)/\Tors{\tilde{H}_n(F)}
\to H_n(F, K)/\Tors{H_n(F, K)}
\to \tilde{H}_{n-1}(K) \to 0
$$
is well defined and exact. 

(3) If $K$ is a homotopy sphere, 
then $H_n(K) = 0 = \tilde{H}_{n-1}(K)$, and hence
$$0 \to \tilde{H}_n(F) \to H_n(F, K) \to 0$$
is exact. Thus the result is obvious. This completes the proof. 
\end{proof}

The following can be regarded as a correction of
\cite[Proposition~2.1]{B2}.

\begin{prop}\label{prop2.1}
Let $K$ be an exact $(2n-1)$-link, $n \geq 2$, 
and $A$ its Seifert form associated with an exact
Seifert surface. Then, there exists a simple
$(2n-1)$-link $K'$ cobordant to $K$ such that
the Seifert form of $K'$ associated with an $(n-1)$-connected
Seifert surface is algebraically cobordant to $A$.
\end{prop}

\begin{rmk}
Note that when $n = 1$, every $1$-link
admits a connected Seifert surface, and
hence is simple.
\end{rmk}

\begin{proof}[Proof of Proposition~\textup{\ref{prop2.1}}]
Let $F$ be an exact Seifert surface of $K$.
By exactly the same method as in \cite{B2, L1},
with the help of an engulfing theorem,
we can perform embedded surgeries on $F$
inside the disk $D^{2n+2}$ along spheres $a$ of
dimensions $\leq n-1$ embedded in $F$ so that
we obtain a simple knot $K'$ cobordant to $K$
and an $(n-1)$-connected Seifert surface $F'$
for $K'$.

Let us examine the relationship between the Seifert
forms with respect to $F$ and $F'$.
If the sphere $a$ along which the surgery is
performed is of dimension less than or equal to $n-2$,
then it does not affect the $n$-th homology of $F$.
We again denote by $F$ the result of such surgeries:
in particular, $F$ is $(n-2)$-connected.
Let us now consider the case where $a$ is of
dimension $n-1$. In the following, $[a]$ will
denote the homology class in $H_{n-1}(F)$ represented
by $a$, where we fix its orientation once and for all.

\medskip

{\it Case}~1. When $[a]$ has infinite order in $H_{n-1}(F)$.

Since $K$ is exact, the boundary homomorphism
$\partial_* : H_n(F, K) \to H_{n-1}(K)$ is surjective.
By the exact sequence
$$H_n(F, K) \spmapright{\partial_*} 
H_{n-1}(K) \spmapright{i_*} H_{n-1}(F)
\spmapright{j_*} H_{n-1}(F, K),$$
where $i : K \to F$ and $j : F \to (F, K)$ are the inclusions,
we see that $j_*$ is injective and hence
$j_*[a]$ has infinite order in $H_{n-1}(F, K)
\cong H^{n+1}(F)$.
Therefore, there exists an $(n+1)$-cycle $\tilde{a}$
of $F$ such that the intersection number $a \cdot \tilde{a}$
does not vanish. We choose $\tilde{a}$ so that
$m = |a \cdot \tilde{a}| (> 0)$ is the smallest possible.

Let $\psi : D^n \times D^{n+1} \to D^{2n+2}$ be
the $n$-handle used by the surgery in question
such that $\psi(S^{n-1} \times \{0\}) = a$.
As in \cite{B2}, let us put
$$F_T = F \setminus \Int(\psi(S^{n-1} \times D^{n+1})),
\quad F^\star = F_T \cup \psi(D^n \times S^n).$$
Let us consider the Mayer-Vietoris exact sequence
associated with the
decomposition $F = F_T \cup \psi(S^{n-1} \times D^{n+1})$:
\begin{eqnarray*}
& & H_{n+1}(F) \spmapright{s} 
H_n(\psi(S^{n-1} \times S^n)) \to
H_n(F_T) \to H_n(F) \\
& & \spmapright{t} H_{n-1}(\psi(S^{n-1} \times S^n))
\spmapright{u} 
H_{n-1}(F_T) \oplus H_{n-1}(\psi(S^{n-1} \times D^{n+1})).
\end{eqnarray*}
Since the map $s$ is given by the intersection number with
$a$, its image coincides with $m \Z \subset \Z
\cong H_n(\psi(S^{n-1} \times S^n))$. 
Furthermore, since $u$ is an injection, $t$
is the zero map. Therefore, we have the exact sequence
$$0 \to \Z_m \to H_n(F_T) \to H_n(F) \to 0.$$
Therefore, the inclusion $F_T \to F$
induces an isomorphism
$$H_n(F_T)/\Tors{H_n(F_T)} \to H_n(F)/\Tors{H_n(F)}.$$

\begin{rmk}
In \cite{B2}, it is stated that the map $s$
is surjective, since the image is generated by
the intersection of the $(n+1)$-cycle dual to $a$
and $\psi(S^{n-1} \times S^n)$. However, such
an $(n+1)$-cycle dual to $a$ may not exist, since
$\partial F$ may not be a sphere. Here, we used the
assumption that $F$ is an exact Seifert surface
in order to show that the map $s$ is non-trivial.
\end{rmk}

Similarly we also have the following exact sequence
obtained from the Mayer-Vietoris
exact sequence associated with the decomposition
$F^\star = F_T \cup \psi(D^n \times S^n)$:
\begin{eqnarray*}
& & 0 \to H_n(\psi(S^{n-1} \times S^n)) \to H_n(F_T)
\oplus H_n(\psi(D^n \times S^n)) \to H_n(F^\star) \\
& & \to H_{n-1}(\psi(S^{n-1} \times S^n)) \spmapright{u'}
H_{n-1}(F_T).
\end{eqnarray*}
Note that the map $u'$ is injective, since the
image of the composition 
$$H_{n-1}(\psi(S^{n-1} \times S^n)) \spmapright{u'}
H_{n-1}(F_T) \spmapright{v} H_{n-1}(F)$$
is generated by $[a]$ which is of infinite order,
where $v$ is the homomorphism induced by the
inclusion. Therefore, we see that the inclusion
induces an isomorphism $H_n(F_T) \to H_n(F^\star)$.

Summarizing, we have the isomorphisms
$$H_n(F)/\Tors{H_n(F)} \spmapleft{\cong}
H_n(F_T)/\Tors{H_n(F_T)} \spmapright{\cong}
H_n(F^\star)/\Tors{H_n(F^\star)}$$
induced by the inclusions.

\medskip

{\it Case}~2. When $[a]$ has finite order in $H_{n-1}(F)$.

Let us denote the order of $[a]$ by $p > 0$.
There exists an $n$-chain $\sigma$ in $F$ such that
$\partial \sigma = p a$. We may assume that
$\sigma$ does not intersect with $a$ outside of
its boundary. Then, we have an $n$-chain
$\sigma'$ in $F_T$ such that $[\partial \sigma']
= p[\psi(S^{n-1} \times \{\ast\})]$ in
$H_{n-1}(\psi(S^{n-1} \times S^n))$.

As before, we have the following exact sequence:
$$
H_n(\psi(S^{n-1} \times S^n)) \spmapright{w} H_n(F_T) \to
H_n(F) \to 0.
$$
Since $[\psi(\{\ast\} \times S^n)] \in H_n(F_T)$ 
has non-zero intersection number with the homology
class in $H_n(F_T, \partial F_T)$ represented by $\sigma'$,
we see that the map $w$ above is injective.
Note that then $(\Image w)^{\wedge}$ is infinite cyclic.
Let a generator of $(\Image w)^{\wedge}$ be
denoted by $\ell \in H_n(F_T)$. 
Then, we have the following exact sequence:
$$0 \to \Z\langle \ell \rangle \to H_n(F_T)/\Tors{H_n(F_T)}
\to H_n(F)/\Tors{H_n(F)} \to 0,$$
where $\Z\langle \ell \rangle$ denotes the
infinite cyclic group generated by $\ell$.
This implies that $H_n(F_T)/\Tors{H_n(F_T)}
\cong (H_n(F)/\Tors{H_n(F)}) \oplus \Z\langle \ell \rangle$.

Similarly, we have the exact sequence
\begin{eqnarray*}
& & 0 \to H_n(\psi(S^{n-1} \times S^n)) \to
H_n(F_T) \oplus H_n(\psi(D^n \times S^n)) \to
H_n(F^\star) \\
& & \spmapright{t'} H_{n-1}(\psi(S^{n-1}
\times S^n)) \spmapright{u'} 
H_{n-1}(F_T) \to H_{n-1}(F^\star) \to 0.
\end{eqnarray*}
The image of $p$ times the generator of $H_{n-1}
(\psi(S^{n-1} \times S^n))$ by $u'$ vanishes,
since it bounds $\sigma'$ in $F_T$. On the
other hand, if $p'$ times the generator
belongs to $\Ker{u'}$ for some $p'$ with $0 < p' < p$,
then the order of $[a]$ is strictly less than $p$,
which is a contradiction. Therefore, the image
of $s'$ is generated by $z = p[\psi(S^{n-1} \times \{\ast\})]$.
Hence, we have the exact sequence
$$0 \to H_n(F_T) \to H_n(F^\star) 
\spmapright{t'} \Z\langle z \rangle \to 0,$$
where $\Z\langle z \rangle$ is the infinite cyclic group
generated by $z \in H_{n-1}(\psi(S^{n-1} \times S^n))$.
Let $\eta^\star$ be the $n$-cycle in $F^\star$
obtained by the union of $p$ times $\psi(D^n \times \{\ast\})$
and $\sigma'$. Set $\ell^\star = [\eta^\star] \in
H_n(F^\star)$.
Then the image of $\ell^\star$ by $t'$
coincides with $\pm z$. Therefore, we see that
$H_n(F^\star)/\Tors{H_n(F^\star)} \cong (H_n(F_T)/\Tors{H_n(F_T)})
\oplus \Z\langle \ell^\star \rangle$.

Summarizing, we have
$$H_n(F^\star)/\Tors{H_n(F^\star)} \cong
(H_n(F)/\Tors{H_n(F)}) \oplus \Z \langle \ell \rangle
\oplus \Z \langle \ell^\star \rangle.$$
So, in this case, the rank of the $n$-th
homology group increases by two as a result
of the surgery.

\medskip

In the following, we denote by $F$ the original
Seifert surface for $K$ and by $F'$ the $(n-1)$-connected
Seifert surface for $K'$ obtained as a result
of the surgeries. Set $G = H_n(F)/\Tors{H_n(F)}$.
Note that
\begin{equation}
G' = H_n(F')/\Tors{H_n(F')} \cong
G \oplus \left(\oplus_{i \in \mathcal{I}}
\left(\Z \langle \ell_i \rangle \oplus \Z \langle \ell_i^\star \rangle
\right)\right),
\label{iso}
\end{equation}
where the indices in $\mathcal{I}$ correspond to
the surgeries necessary to kill the torsion of the
$(n-1)$-th homology, and $\ell_i$ (or $\ell_i^\star$)
corresponds to the generator $\ell$ (resp.\ $\ell^\star$)
above (see Case~2).

Let $A$ (or $A'$) be the Seifert form for $F$
(resp.\ $F'$) defined on
$H_n(F)/\Tors{H_n(F)}$ (resp.\ $H_n(F')/\Tors{H_n(F')}$).
Furthermore, let $S$ (or $S'$) be the intersection form
of $F$ (resp.\ $F'$).
Note that $\Ker{S^*} \cong H_n(K)$ 
corresponds to $\Ker{(S')^*} \cong H_n(K')$
under the isomorphism (\ref{iso}).

Set $B = (-A) \oplus A'$ and
$S_B = (-S) \oplus S'$, which are bilinear forms
defined on $G \oplus G'$. Note that $G$ can
be identified with a submodule of $G'$ under the
isomorphism (\ref{iso}).
Let $M$ be the submodule of
$G \oplus G'$
generated by the elements of the form $(x, x)$ with
$x \in G$ and by $\ell_i$, $i \in \mathcal{I}$.

As in \cite{B2}, we see easily that
$M$ is a metabolizer for $B$.
Furthermore, $\overline{M}$ is pure in $\overline{G \oplus G'}$
and we can easily check that
$$M \cap \Ker{S_B^*} = \{(x, x) \in G \oplus G'\,|\,
x \in \Ker{S^*}\}.$$

Let $y$ be an arbitrary nonzero
element of $\Tors{(\Coker{S^*})}$. We
denote the order of $y$ by $q$.
Let 
$$
\partial'_* : G^* = H_n(F, K)/\Tors{H_n(F, K)} \to H_{n-1}(K)
$$
be the homomorphism induced by the boundary homomorphism,
which is well defined and surjective, 
since $F$ is an exact Seifert surface.
Furthermore, the map 
$$
H_n(F)/\Tors{H_n(F)} \to H_n(F, K)/\Tors{H_n(F, K)}
$$ 
induced by the inclusion
is identified with
$S^*$ by virtue of the Poincar\'e duality,
and its image coincides with $\Ker{\partial'_*}$.
(We also have similar statements for $(S')^*$ as well.)

Thus, there exists a $\tilde{y} \in G^*$ such that 
$\partial'_* \tilde{y} = y$ under the identification
$\Coker{S^*} = H_{n-1}(K)$.
Then, $q(\tilde{y}, \tilde{y}) \in G^* \oplus (G')^*$
lies in $S_B^*(M)$, which implies that
$(\tilde{y}, \tilde{y}) \in G^* \oplus (G')^*$
lies in $S_B^*(M)^\wedge$.
Therefore, we have
\begin{equation}
d(S_B^*(M)^\wedge) \supset \{(y, y)\,|\,
y \in \Tors{(\Coker{S^*})}\}
\label{inclusion}
\end{equation}
under the natural identification
$$\Coker{S^*} = H_{n-1}(K) = H_{n-1}(K')
= \Coker{(S')^*}.$$

\begin{lem}
The order of
$d(S_B^*(M)^\wedge)$ coincides with
that of $\Tors{(\Coker{S^*})}$.
\end{lem}

\begin{proof}
Since $S_B^*(M)$ is of finite index
in $S_B^*(M)^\wedge$, we can write
$$
S_B^*(M)^\wedge/S_B^*(M) \cong
\oplus_{i=1}^k \Z_{a_i}, 
$$
where $a_i$ are positive integers such that
$a_i$ divides $a_{i+1}$ for all $i = 1, 2, \ldots, k-1$,
and $k = \rk{S_B^*(M)^{\wedge}}$.
(Here, we do not exclude the case where $a_1 = \cdots
= a_r = 1$ for some $r$ with $1 \leq r \leq k$.)

Since $\overline{M}$ is pure in $\overline{G \oplus G'}$,
we have $S_B^*(G \oplus G') \cap S_B^*(M)^\wedge = S_B^*(M)$
by \cite[\S2]{BM}. Therefore, the quotient
map $d : G^* \oplus (G')^* \to \Coker{S_B^*}$
restricted to $S_B^*(M)^\wedge$ can
be identified with the quotient map
$S_B^*(M)^\wedge \to S_B^*(M)^\wedge/S_B^*(M)$.

Let us consider $\overline{S_B} : \overline{G \oplus G'} 
\times
\overline{G \oplus G'} \to \Z$, the $\varepsilon$-symmetric
non-degenerate bilinear form induced from $S_B$
on $\overline{G \oplus G'} = (G \oplus G')/\Ker{S_B^*}$.
Since $\overline{M}$ is pure in $\overline{G \oplus G'}$,
we have a submodule $N$ of $\overline{G \oplus G'}$
such that $\overline{G \oplus G'} = \overline{M} \oplus N$.
Note that $S_B^*(M)^\wedge/S_B^*(M)$ is naturally
isomorphic to $\overline{S_B}^*(\overline{M})^\wedge/
\overline{S_B}^*(\overline{M})$. Therefore,
by taking appropriate bases of $\overline{M}$
and $N$, we may assume that a matrix representative
of $\overline{S_B}$ is of the form
$$\begin{pmatrix}
0 & D \\
\varepsilon D^T & \ast
\end{pmatrix},$$
where $D$ is the $k \times k$ diagonal matrix
with diagonal entries $a_1, a_2, \ldots, a_k$.
In particular, the order of 
$$\Tors{(\Coker{S_B^*})} =
\Coker{\overline{S_B}^*} = 
\overline{G \oplus G'}^*/\overline{S_B}^*(\overline{G
\oplus G'})$$
is equal to $(a_1 a_2 \cdots a_k)^2$.

Note that
$$\overline{S_B}^*(\overline{M})^\wedge/
\overline{S_B}^*(\overline{M}) \cong 
S_B^*(M)^\wedge/S_B^*(M) \cong \oplus_{i=1}^k
\Z_{a_i}.$$
Therefore, the order of 
$$\Coker{\overline{S_B}^*} = \Coker{\overline{S}^*}
\oplus \Coker{\overline{S'}^*}$$ 
coincides with the square
of the order of $\overline{S_B}^*(\overline{M})^\wedge/
\overline{S_B}^*(\overline{M})$.
Therefore, we have the lemma.
\end{proof}

Combining the above lemma with (\ref{inclusion}),
we have
$$d(S_B^*(M)^\wedge) = \{(y, y)\,|\,
y \in \Tors{(\Coker{S^*})}\}.$$

Therefore, we conclude that $A$ and $A'$ are
algebraically cobordant.
This completes the proof of Proposition~\ref{prop2.1}.
\end{proof}

As a corollary, we have the following, which can
be regarded as a correction of \cite[Proposition~2.2]{B2}.

\begin{prop}\label{prop2.2}
Let $K$ be an exact $(2n-1)$-link, $n \geq 3$, 
and $A$ its Seifert form associated with an exact
Seifert surface. Then, there exists a simple
$(2n-1)$-link $K'$ cobordant to $K$ such that
the Seifert form of $K'$ associated with an $(n-1)$-connected
Seifert surface coincides with $A$.
\end{prop}

\begin{proof}
By Proposition~\ref{prop2.1}, there exists
a simple $(2n-1)$-link $K''$ cobordant to $K$
such that the Seifert form $A''$ of $K''$ associated
with an $(n-1)$-connected Seifert surface is algebraically
cobordant to $A$. On the other hand, it is known that
there exists
a simple $(2n-1)$-link $K'$ whose Seifert form associated with
an $(n-1)$-connected Seifert surface coincides with $A$.
Since $A$ and $A''$ are algebraically cobordant,
we see that $K'$ and $K''$ are cobordant by \cite{BM}.
Then, $K$ and $K'$ are cobordant, and the desired result
follows.
\end{proof}

\begin{rmk}
We do not know if the above proposition
holds also for $n = 2$ or not.
\end{rmk}

For exact links, we have the following,
which can be regarded as a correction of
\cite[Th\'eor\`eme~1]{B2}.

\begin{thm}\label{thm1}
Let $K$ and $K'$ be exact $(2n-1)$-links, $n \geq 3$.
If their Seifert forms with respect to exact
Seifert surfaces are algebraically cobordant,
then $K$ and $K'$ are cobordant.
\end{thm}

\begin{proof}
By Proposition~\ref{prop2.2},
there exists a simple $(2n-1)$-knot $\tilde{K}$
(or $\tilde{K}'$) cobordant to $K$ (resp.\ $K'$)
such that the Seifert form of $\tilde{K}$ (resp.\ $\tilde{K}'$)
with respect to an $(n-1)$-connected
Seifert surface coincides with the
Seifert form of $K$ (resp.\ $K'$) with respect
to an exact Seifert surface. By our assumption,
the Seifert forms of $\tilde{K}$ and
$\tilde{K}'$ are algebraically cobordant.
Then, by \cite{BM}, we see that $\tilde{K}$
and $\tilde{K}'$ are cobordant. Therefore,
$K$ and $K'$ are cobordant.
\end{proof}

\section{Cobordism of fibered knots}

The following can be regarded as a
correction of \cite[Th\'eor\`emes~2 et A]{B2}.

\begin{thm}
Let $K$ and $K'$ be two fibered $(2n-1)$-links,
$n \geq 3$. Then, $K$ and $K'$ are cobordant
if and only if their Seifert forms with respect
to their fibers are algebraically cobordant.
\end{thm}

\begin{proof}
By Lemma~\ref{lem:exact}, a fiber
of a fibered link is always exact.
Thus, by Theorem~\ref{thm1},
if the Seifert forms with respect to the
fibers are algebraically cobordant,
then $K$ and $K'$ are cobordant.

Conversely, suppose that $K$ and $K'$ are
cobordant. Let $A$ (or $A'$) be the
Seifert form of $K$ (resp.\ $K'$)
with respect to a fiber.
By Proposition~\ref{prop2.2}
and Lemma~\ref{lem:exact},
there exists a simple $(2n-1)$-link
$\tilde{K}$ (or $\tilde{K}'$) cobordant to $K$
(resp.\ $K'$) such that the Seifert form
with respect to an $(n-1)$-connected
Seifert surface coincides with $A$ (resp.\ $A'$).
Since $A$ and $A'$ are unimodular, we see
that $\tilde{K}$ and $\tilde{K'}$ are
fibered (for example, see \cite{D, Kato}).
Since $K$ and $K'$ are cobordant,
we see that $\tilde{K}$ and $\tilde{K}'$ are
also cobordant. Then, by \cite{BM},
we see that $A$ and $A'$ are algebraically
cobordant.
This completes the proof.
\end{proof}

%\begin{rmk}
%As to \cite[Proposition~2.4]{B2}, it seems that
%the proof contains a gap. In fact, even if $S|_H$
%is unimodular, we cannot find $d_i^*$ as in
%\cite[\S4.2]{B2}. For example, if $S|_H$ is
%indecomposable, then such $d_i^*$ do not exist.
%Accordingly, \cite[Proposition~2.5]{B2}
%also needs to be corrected.
%\end{rmk}

%%%%%%%%%%%%%%% some parts deleted 

We finish with an open problem.

\begin{ask}
Does there exist
a link which is not exact?
\end{ask}


\begin{thebibliography}{99}
%
%\bibitem{B}V.~Blanl\oe il, 
%Cobordisme des entrelacs fibr\'es simples et forme de Seifert, 
%C.\ R.\ Acad.\ Sci.\ Paris, Ser.\ I, Math.\ {\bf 320} (1995), 
%985--988.
%
\bibitem{B2}V.~Blanl\oe il,
Cobordisme des entrelacs,
Ann.\ Fac.\ Sci.\ Toulouse Math.\ (6) \textbf{7} (1998), 
185--205. 
%
\bibitem{BM}V.~Blanl\oe il and F.~Michel,
A theory of cobordism for non-spherical links,
Comment.\ Math.\ Helv.\ {\bf 72} (1997), 30--51.
%
\bibitem{BS2}V.~Blanl\oe il and O.~Saeki, 
Cobordism of fibered knots and related topics,
in ^^ ^^ Singularities in geometry and topology 2004",
Adv.\ Stud.\ Pure Math.\ \textbf{46}, Math.\ Soc.\ Japan, Tokyo, 2007,
pp.~1--47.
%
\bibitem{D}A.~Durfee,
Fibered knots and algebraic singularities,
Topology {\bf 13} (1974), 47--59.
%
%\bibitem{Hirsch}M.W.~Hirsch,
%Embeddings and compressions of polyhedra and smooth manifolds,
%Topology \textbf{4} (1966), 361--369. 
%
\bibitem{Kato}M.~Kato,
A classification of simple
spinnable structures on a $1$-connected Alexander manifold,
J.\ Math.\ Soc.\ Japan {\bf 26} (1974), 454--463.
%
%\bibitem{Kauff}L.~Kauffman,
%Branched coverings, open books and
%knot periodicity,
%Topology {\bf 13} (1974), 143--160.
%
\bibitem{L1}J.~Levine, Knot cobordism groups in 
codimension two, 
Comment.\ Math.\ Helv.\ \textbf{44} (1969),  
229--244.
%
%\bibitem{LM}
%S.~L\'opez de Medrano, 
%Invariant knots and surgery in codimension $2$, 
%Actes du Congr\`es International des Math\'ematiciens 
%(Nice, 1970), Tome 2, pp.~99--112,
%Gauthier-Villars, Paris, 1971. 
%
%\bibitem{Matsumoto}Y.~Matsumoto,
%Note on the splitting problem in codimension two,
%preprint.
%
%\bibitem{MH}J.~Milnor and D.~Husemoller, 
%Symmetric bilinear forms, Ergebnisse der Mathematik und ihrer 
%Grenzgebiete, Band 73, Springer-Verlag, New York, Heidelberg, 1973. 
%
%\bibitem{V1}R.~Vogt, Cobordismus von Knoten, 
%in ^^ ^^ Knot theory" (Proc.\ Sem., Plans-sur-Bex, 1977), 
%pp.~218--226, Lecture Notes in Math., 685, Springer, Berlin, 1978.
%
\bibitem{V2}R.~Vogt, Cobordismus von hochzusammenh\"angenden 
Knoten, Dissertation, Rheinische Friedrich-Wilhelms-Universit\"at, 
Bonn, 1978, Bonner Mathematische Schriften, 116,
Universit\"at Bonn, Mathematisches Institut, Bonn, 1980. 

\end{thebibliography}
\end{document}